\documentclass[10pt]{article}
\usepackage{a4,amsmath,amssymb,amsthm,graphicx,enumerate,color}

\newtheorem*{theorem*}{Theorem}
\newtheorem{lem}{Lemma}
\newtheorem{prop}{Proposition}
\newtheorem{corollary}{Corollary}
\theoremstyle{definition}
\newtheorem{dfn}{Definition}
\newtheorem{remark}{Remark}

\numberwithin{equation}{section}
\newcommand{\R}{\ensuremath{\mathbf{R}}}
\newcommand{\defn}{\ensuremath{\overset{\mathrm{def}}{=}}}
\newcommand{\dd}{\ensuremath{\mathrm{d}}}

\begin{document}
\title{\bf{A Novel Method of Solution For \\ The Fluid Loaded Plate}}
\author{A.C.L Ashton\footnote{a.c.l.ashton@damtp.cam.ac.uk}\, and A.S Fokas\footnote{t.fokas@damtp.cam.ac.uk}\\
\\
Department of Applied Mathematics \\
and Theoretical Physics\\
University of Cambridge \\
CB3 0WA, UK}
\maketitle
\begin{abstract}
We study the equations governing a fluid loaded plate. We first reformulate these equations as a system of two equations, one of which is an explicit non-local equation for the wave height and the velocity potential on the free surface. We then concentrate on the linearised equations and show that the problems formulated either on the full or the half line can be solved by employing the unified approach to boundary value problems introduced by on of the authors in the late 1990's. The problem on the full line was analysed by Crighton et. al. using a combination of Laplace and Fourier transforms. The new approach avoids the technical difficulty of the a priori assumption that the amplitude of the plate is in $L^1_{\dd t}(\mathbf{R}^+)$ and furthermore yields a simpler solution representation which immediately implies that the problem is well-posed. For the problem on the half-line, a similar analysis yields a solution representation, which however, involves two unknown functions. The main difficulty with the half-line problem is the characterisation of these two functions. By employing the so-called global relation, we show that the two functions can be obtained via the solution of a complex valued integral equation of the convolution type. This equation can be solved in closed form using the Laplace transform. By prescribing the initial data $\eta_0$ to be in $H^3_{\langle 3\rangle}(\mathbf{R}^+)$, or equivalently twice differentiable with sufficient decay at infinite, we show that the solution depends continuously on the initial data, and hence, the problem is well-posed.
\end{abstract}

\newpage
\section{Introduction}
\subsection{The Governing Equations}
The problem we study concerns the motion of a (semi) infinite elastic plate lying in the plane $y=0$, driven from below by a uniform flow with velocity $U$ in the $x$-direction. We denote the amplitude of the plate by $\eta (x,t)$, and use $\phi(x,y,t)$ to denote the potential function for the perturbation from the mean flow. The surface of the plate is described by $\mathcal{S}_\eta$:
\[ \mathcal{S}_\eta = \{ (x,y) \in \mathbf{R}^2: y=\eta(x,t)\}. \]
We use $N_\mathcal{S}=(-\eta_x,1)$ to denote the upward normal to $\mathcal{S}_\eta$ and $\Omega_\eta$ to denote the region occupied by the fluid in $y<\eta$. In the case of the half-line problem, we have $\eta\equiv 0$ in $x\leq 0$. The evolution of the plate is governed by the beam equation and the kinematic boundary condition:
\begin{alignat}{2} 
\eta_{tt} + \eta_{xxxx} &= p&\qquad &\textrm{on $\mathcal{S}_\eta$,} \label{beam}\\
(U+\phi_x, \phi_y)\cdot N_\mathcal{S} &= \eta_t&\qquad &\textrm{on $\mathcal{S}_\eta$,} \label{dynamic}
\end{alignat}
where $p=p(x,y,t)$ denotes the pressure of the fluid in $\Omega_\eta$. In addition, the Bernoulli condition is valid on the plate:
\begin{equation} \phi_t + U\phi_x + \tfrac{1}{2}\phi_x^2 + \tfrac{1}{2}\phi_y^2 + p = 0 \qquad \textrm{on $\mathcal{S}_\eta$}. \label{Bernoulli}\end{equation}
We assume the fluid is incompressible so that the potential is harmonic in $\Omega_\eta$:
\begin{equation} \phi_{xx} + \phi_{yy} = 0 \quad \textrm{in $\Omega_\eta$}. \label{incompressible}\end{equation}
We can eliminate the pressure term using \eqref{beam} and \eqref{Bernoulli}, so the governing equations for the amplitude $\eta$ and the potential $\phi$ become:
\begin{subequations}\label{nonlinear}
\begin{alignat}{2}
\phi_{xx} + \phi_{yy} &= 0&\qquad &\textrm{in $\Omega_\eta$,} \label{incompressible-} \\
\eta_{tt} + \eta_{xxxx}  +\phi_t +U\phi_x + \tfrac{1}{2}\phi_x^2 + \tfrac{1}{2}\phi_y^2&=0&\qquad &\textrm{on $\mathcal{S}_\eta$,} \label{beam-}\\
(U+\phi_x, \phi_y)\cdot N_\mathcal{S} &= \eta_t&\qquad &\textrm{on $\mathcal{S}_\eta$.} \label{dynamic-}
\end{alignat}
\end{subequations}
In addition we assume $|\nabla \phi| \rightarrow 0$ on $\partial\Omega_\eta \setminus \mathcal{S}_\eta$.

Let $\varphi$ denote the potential evaluated on the free surface, i.e,
\begin{equation} \varphi(x,t)=\phi(x,\eta(x,t),t). \label{newcoord}\end{equation}
An application of the chain rule gives
\begin{subequations}\label{chain}
\begin{align} \varphi_x &= \phi_x + \eta_x \phi_y, \label{chain1}\\
 \varphi_t &= \phi_t + \eta_t \phi_y. \label{chain2}
\end{align}
\end{subequations}
Equations \eqref{dynamic-} and \eqref{chain} constitute a non-singular set of equations for $\{\phi_x,\phi_y,\phi_t\}$ on the surface $\mathcal{S}_\eta$. Solving these equations we find:
\begin{subequations}\label{inversion}
\begin{align}
(1+\eta_x^2)\phi_x &= \varphi_x - \eta_x (\eta_t+U\eta_x), \\
(1+\eta_x^2)\phi_y &= \varphi_x \eta_x + \eta_t + U\eta_x, \\
(1+\eta_x^2)\phi_t &= (1+\eta_x^2) \varphi_t - \eta_t (\eta_t + U\eta_x + \eta_x \varphi_x) .
\end{align}
\end{subequations}
We can now write \eqref{beam-} in terms of $(\eta,\varphi)$:
\begin{equation} \eta_{tt} + \eta_{xxxx} +\varphi_t - \tfrac{1}{2}\eta_t^2 -\tfrac{1}{2}U^2 + \frac{(U+\varphi_x - \eta_x\eta_t)^2}{2(1+\eta_x^2)} =0. \end{equation}
In what follows, we reduce equations \eqref{incompressible-} and \eqref{dynamic-} into \emph{one} non-local equation in the coordinates $(\eta,\varphi)$.

\subsection{The Non-Local Formulation}
One of the key ingredients in our approach for the analysis of both the nonlinear and linear problems is the employment of the so-called global relation \cite{fokas2008uab}. The global relation is a direct consequence of the fact that the Laplace equation is equivalent to the following equation:
\begin{equation} \partial_x \Big( e^{-i kx + k y} (  \phi_x -i\phi_y )\Big) + \partial_y\Big( e^{-i kx + k y}( \phi_y+i\phi_x )\Big)  = 0 \qquad \textrm{in $\Omega_\eta$.}\label{totdiv}\end{equation}
Integrating \eqref{totdiv} over $\Omega_\eta$ and employing the divergence theorem gives:
\begin{multline} \int_{\mathcal{S}_\eta} e^{-ikx+ky}(\phi_x -i\phi_y, \phi_y+i\phi_x)\cdot N_{\mathcal{S}}\, \dd x \\
+ \lim_{h\rightarrow \infty} e^{-kh} \int_{y=-h} e^{-ikx}(\phi_x -i\phi_y, \phi_y+i\phi_x)\cdot N_h \, \dd x =0, \label{h1}
\end{multline}
where $N_h=(0,-1)$. By making the restriction $k>0$, the second integral in \eqref{h1} vanishes. Using the kinematic boundary condition \eqref{dynamic-} and equation \eqref{chain1}, the integral equation \eqref{h1} becomes:
\begin{equation} \int_\mathbf{R} e^{-ikx+k\eta} \Big( \eta_t + U\eta_x + i\varphi_x \Big)\dd x =0, \qquad k>0.\end{equation}
We thus have obtained the following result.
\begin{prop}
The solution to the boundary value problem specified in \eqref{nonlinear} is determined by the pair of real valued functions $\eta(x,t),\varphi(x,t)$ which satisfy:
\begin{subequations} \label{total}
\begin{align}
\int_\mathbf{R} e^{-ikx+k\eta} \Big( \eta_t + U\eta_x + i\varphi_x \Big)\dd x &=0, \qquad k>0, \\
\eta_{tt} + \eta_{xxxx} +\varphi_t - \tfrac{1}{2}\eta_t^2 -\tfrac{1}{2}U^2 + \frac{(U+\varphi_x - \eta_x\eta_t)^2}{2(1+\eta_x^2)} &=0,
\end{align}
\end{subequations}
where $\varphi = \phi(x,\eta(x,t),t)$ is the potential on the surface $\mathcal{S}_\eta$ and $\eta(x,t)$ is the amplitude of the plate.
\end{prop}
From this stage onwards we restrict attention to the linearised problem. For the analysis of the nonlinear problem for \emph{a collection} of fluid loaded membranes, we refer the reader to \cite{antnlflp}.

Dropping higher order terms, equations \eqref{total} become:
\begin{subequations}\label{linear}
\begin{align}
\hat{\eta}_t + ikU \hat{\eta} -k\hat{\varphi} &=0, \qquad k>0, \\
\eta_{tt} + \eta_{xxxx} + \varphi_t + U\varphi_x &=0,
\end{align}
\end{subequations}
where the hat denotes the usual Fourier transform.

\subsection{The Cauchy Problems}
We will study the initial-boundary value problem corresponding to the equations \eqref{linear} on both the full and half line. We first work on a formal level to derive explicit solutions to the underlying problem, \emph{then} prove what properties these solutions have. Frequently we will refer to the standard Sobolev space $H^s(X)$, where $X$ is either the real or full line and $H^s(X)=W^{s,2}(X)$:
\[ W^{k,p}(\Omega) \defn \{ \partial^\alpha f\in  L^p(\Omega), 0\leq |\alpha |\leq k \}.\] 
Taking into consideration that the plate is assumed to be thin, it is not obvious that there exists a solution with $\eta \in L^2_{\dd x}(X)$, since disturbances in the plate could propagate at arbitrarily high speeds. In other words, even if $\eta(x,0) \in C^\infty_c (X)$, it does not immediately follow that $\eta \in L^2_{\dd x}(X)$. However, as we shall see, the problem with such a restrictive function class \emph{does} admit a solution. Indeed, it will become apparent that rapid oscillations facilitate the precribed integrability conditions.

\subsubsection{The Full Line}
The governing equations corresponding to the Cauchy problem for the fluid loaded plate on the full line are given by:
\begin{subequations}\label{full_line_eq}
\begin{alignat}{2}
\eta_{tt} + \eta_{xxxx}  + \varphi_t + U\varphi_x &=0 &\quad   & x \in \R,\,\, 0<t<T,  \label{whole1} \\
\hat{\eta}_t + ikU \hat{\eta} -k\hat{\varphi} &=0, && 0<k<\infty, \,\, 0<t<T, \label{wholeFour} \\
\eta (x,0) &= \eta_0 (x), && x \in \R, \label{whole2} \\
\varphi_t (x,0) + U \varphi_x (x,0)&=0,  && x\in\R , \label{whole3}
\end{alignat}
\end{subequations}
where $\eta_0 \in H^2_{\dd x}(\R)$ is the initial profile of the plate.

\subsubsection{The Semi-Infinite Line}
In the case of the semi-infinite line, we impose extra conditions at the hinge which govern the curvature of the plate:
\begin{subequations}\label{half_line_eq}
\begin{alignat}{2}
\eta_{tt} + \eta_{xxxx}+ \varphi_t  + U\varphi_x &=0, &\quad & 0<x<\infty,\,\, 0<t<T, \label{half1}  \\
\hat{\eta}_t + ikU \hat{\eta} -k\hat{\varphi} &=0, && 0<k<\infty, \,\, 0<t<T, \label{halfFour} \\
\eta (x,0) &= \eta_0 (x), && 0<x<\infty, \label{initial} \\
\varphi_t (x,0) + U \varphi_x (x,0) &=0, && 0<x<\infty, \label{half3}\\
\eta (0,t) = \eta_x (0,t) &= 0, && 0<t<T, 
\end{alignat}
\end{subequations}
where the hat denotes the Fourier transform \emph{on the full line}. In this case we assume $\eta_0\in H^2_{\dd x}(\R^+)$ as well as $\eta_0 (0)=\eta_0 '(0)=\eta_0 ''(0)=0$. We denote the corresponding function space $H^2_{\langle 2\rangle}(\R^+) \subset H^2_{\dd x}(\R^+)$ whose elements have zero derivatives at the hinge up to and including the second derivative. Given $f\in H^s_{\langle s\rangle}(\mathbf{R}^+)$, note that the extension operator:
\[ E: H^s_{\langle s\rangle}(\mathbf{R}^+) \rightarrow H^s_{\dd x}(\mathbf{R}) : f\mapsto  E f = \begin{cases} f,& x>0 \\
0,& x\leq 0\end{cases} \]
is continuous, i.e. $\| Ef\|_{H^s(\mathbf{R})} \leq C_{s}\|f\|_{H^s_{\langle s\rangle}(\mathbf{R}^+)}$, hence we may regard $\eta_0 \in H^2_{\dd x}(\mathbf{R})$ with $\mathrm{supp} \,\eta \subset \mathbf{R}^+$.

Equations \eqref{full_line_eq} are solved in section \S2. By employing the unified approach for analysing initial-boundary value problems introduced by one of the authors \cite{fokas2008uab} we find expressions for $\hat{\eta}$ and $\hat{\varphi}$. The issue of well-posedness can then be addressed: we demonstrate that the solution to the IBVP in \eqref{full_line_eq} depends continuously on the initial data, with respect to the $L^2_{\dd x}(\R)$ topology, and so the problem is well-posed in the Hadamard sense.

In the third and fourth sections, we concentrate on the problem on the half-line described by equations \eqref{half_line_eq}. The analysis is similar, but there are two main differences: (a) Using an analytic continuation argument, we show that $k$, instead of being restricted to be positive, satisfies the less sturgent condition that it lies within the fourth quadrant of the complex plane. (b) The solution representation involves the two \emph{unknown} functions $\eta_{xx}(0,t)$ and $\eta_{xxx}(0,t)$.

By utilising the extra freedom in $k$ (see (a) above) and employing the approach of \cite{fokas2008uab} we can determine the two unknown functions. Using an appropriate function space we show that the two unknown functions can be determined \emph{uniquely} in terms of the given data. In contrast to the problem on the full line, we find that for $\eta_0 \in H^\sigma_{\langle\sigma\rangle}(\mathbf{R}^+)$, $\sigma <3$, the solution \emph{does not} depend continuously on the initial data, and in this sense, is ill-posed in $L^2_{\dd x}(\R^+)$; however, for $\eta \in H^3_{\langle 3\rangle}(\mathbf{R}^+)$, the problem is well-posed.

\section{The Infinite line}
We use the Fourier transform pair:
\begin{equation}
\eta (x,t) = \frac{1}{\sqrt{2\pi}} \int_{\R} e^{ikx} \hat{\eta}(k,t)\, \mathrm dk, \quad \hat{\eta}(k,t) = \frac{1}{\sqrt{2\pi}}\int_\R e^{-ikx}\eta(x,t)\, \mathrm dx,\label{fouriereta}
\end{equation}
where the integrals are understood in the Lebesgue sense and as such, equalities that follow from inversion theorems are to be understood almost everywhere. Our assumptions about $\eta(x,t)$ implies that $\hat{\eta}$ is well-defined. Similarly, we define the Fourier transform pair for the potential $\varphi (x,t)$:
\begin{equation}
\varphi (x,t) = \frac{1}{\sqrt{2\pi}} \int_\R e^{ikx} \hat{\varphi}(k,t)\, \mathrm dk, \quad \hat{\varphi}(k,t) = \frac{1}{\sqrt{2\pi}}\int_\R e^{-ikx}\varphi (x,t)\, \mathrm dx.\label{fourierphi}
\end{equation}
The reality of $\eta$ implies that the knowledge of $\hat{\eta}$ for $k>0$ is sufficient for the reconstruction of $\eta$:
\begin{equation}\int_\R e^{ikx}\hat{\eta}(k,t)\, \mathrm dk = \int_{\R^+}\left( e^{ikx}\hat{\eta}(k,t) + e^{-ikx} \overline{\hat{\eta}(k,t)} \right) \, \mathrm dk. \label{realeta}\end{equation}
The following proposition provides an equivalent initial-value problem to \eqref{full_line_eq} for $\hat{\eta}$, from which we can reconstruct $\varphi(x,t)$ and $\eta(x,t)$.
\begin{prop}\label{ibvpwholeprop}
The initial-boundary value problem stated in \eqref{full_line_eq} is equivalent to the following initial value problem in the spectral (Fourier) space:
\begin{subequations}\label{easywhole}
\begin{alignat}{2} \left(1+\tfrac{1}{k}\right)\hat{\eta}_{tt} + 2iU\hat{\eta}_t + \left( k^4 - U^2k\right) \hat{\eta} &= 0, &\quad &k> 0,\,\, 0<t<T, \\ 
\tfrac{1}{k}\hat{\eta}_{tt}(k,0) + 2iU \hat{\eta}_t(k,0) - kU^2 \hat{\eta}(k,0) &= 0,& \quad & k> 0, \\
\hat{\eta}(k,0) &= \hat{\eta}_0 (k),& \quad & k> 0. \end{alignat}
\end{subequations}
\end{prop}
For the sake of brevity, we refer the reader to the proof of a similar result in proposition \ref{propo}. We next introduce some useful notation that shall be used throughout.
\begin{dfn}\label{defs}
The functions $\{ c_+ (k), c_- (k), \alpha (k)\}$ are defined as follows:
\begin{equation}c_\pm (k) \defn \pm\left( \frac{\omega_\mp ^2 -k^4}{\omega_-^2 - \omega_+^2}\right)\hat{\eta}_0(k), \qquad \alpha (k) \defn \frac{1}{i(\omega_- -\omega_+)}\label{cpm}\end{equation}
where $\omega_\pm$ are defined as the roots of the dispersion relation $D(\omega,k) = 0$:
\begin{equation}D(k,\omega) \defn  -\left(1+\tfrac{1}{k}\right)\omega^2 + 2U\omega- \left( k^4 - U^2 k\right). \end{equation}
The roots are given explicitly by
\begin{equation} \omega_\pm = \frac{U \pm k^2 \left(
    1+\tfrac{1}{k}-\tfrac{U^2}{k^3}\right)^{1/2}}{\left(1+\frac{1}{k}\right)}. \label{omegapm}\end{equation}
\end{dfn}
Employing this notation, we construct the solution to \eqref{easywhole} and hence \eqref{full_line_eq}.
\begin{prop}
The solution to the IBVP posed in \eqref{full_line_eq} is given by
\begin{align*}
\eta(x,t) &= \frac{1}{\sqrt{2\pi}} \int_{\R^+} \left( e^{ikx}\hat{\eta}(k,t) + e^{-ikx} \overline{\hat{\eta}(k,t)} \right) \, \mathrm dk, \\
\varphi (x,t) &= \frac{1}{\sqrt{2\pi}} \int_{\R^+} \left( e^{ikx}\hat{\varphi}(k,t) + e^{-ikx} \overline{\hat{\varphi}(k,t)} \right) \, \mathrm dk.
\end{align*}
The function $\hat{\eta}(k,t)$ is defined by
\begin{equation} \hat{\eta}(k,t) = c_+ (k) e^{-i\omega_+ (k) t} + c_- (k) e^{-i\omega_- (k)t},\qquad k>0, \label{solnwhole} \end{equation}
where $c_\pm$ and $\omega_\pm$ are defined in \eqref{cpm} and \eqref{omegapm}, and $\hat{\varphi}(k,t)$ defined by \eqref{wholeFour}.
\end{prop}
\begin{proof}
All that is needed is to show that $\hat{\eta}(k,t)$ satisfies the IBVP in proposition \ref{ibvpwholeprop}. This follows routinely by using the definitions and the fact that $\omega_{\pm}(k)$ satisfy the dispersion relation given in definition \eqref{fouriereta}.
\end{proof}
In what follows we discuss the well-posedness of the IBVP \eqref{full_line_eq}. The previous work establishes existence; uniqueness follows from proposition \ref{ibvpwholeprop} using standard uniqueness results for ODEs and the usual results from Fourier analysis on $L^2_{\dd x}(\R)$. The continuous dependence of the solution on the initial data is proven in the next proposition.
\begin{prop}\label{firstwellposed}
The solution to the Cauchy problem \eqref{full_line_eq} is well-posed. The map $S_t:H^2_{\dd x}(\mathbf{R})\rightarrow L^2_{\dd x}(\mathbf{R})$ defined by $\eta_0 \mapsto \eta$ is continuous.
\end{prop}
\begin{proof}
It suffices to show that $\eta(x,t)$ depends continuously in $\eta_0(x)$ with respect to the $L^2_{\dd x}(\R)$ topology, which is straightforward: Denoting $c_\pm (k)$ by $\tilde{c}_\pm(k)\hat{\eta}_0(k)$, it follows from \eqref{solnwhole} that
\begin{align}\| \eta (\cdot, t) \|_{L^2_{\dd x}} &= \| \left(\tilde{c}_-e^{-i\omega_- t} + \tilde{c}_+ e^{-i\omega_+ t}\right)\hat{\eta}_0 \|_{L^2_{\dd k}},\nonumber \\
&= \left\|\left(\frac{\tilde{c}_-e^{-i\omega_- t} + \tilde{c}_+ e^{-i\omega_+ t}}{1+|k|^2}\right)(1+|k|^2)\hat{\eta}_0\right\|_{L^2_{\dd k}}  \label{estim}\end{align}
which follows from Parseval's theorem. Definitions \ref{defs} and some algebra yield:
\begin{equation} |\tilde{c}_-e^{-i\omega_- t} + \tilde{c}_+ e^{-i\omega_+ t}|^2 \equiv  \left| \cos \left(\tfrac{kQt}{1+k}\right) + \tfrac{1}{2ik}\left( \tfrac{Q}{U} + \tfrac{k^2U}{Q}\right)\sin\left(\tfrac{kQt}{1+k}\right)\right|^2, \label{rearrang}\end{equation}
where $Q\equiv \sqrt{k(k^3+k^2-U^2)}$. It is straightforward to prove that
\[ \left\|\frac{\tilde{c}_-e^{-i\omega_- t} + \tilde{c}_+ e^{-i\omega_+ t}}{1+|k|^2}\right\|_{L^\infty} \]
is finite, so H\"older's inequality in the estimate \eqref{estim} gives:
\begin{equation} \| \eta (\cdot, t) \|_{L^2_{\dd x}} \leq c_1  \|(1+|k|^2)\hat{\eta}_0\|_{L^2_{\dd k}}.\label{sobest}\end{equation}
Recalling that $u\in H^s_{\dd x}(\mathbf{R})$ if and only if $(1+|k|^2)^{s/2}\hat{u} \in L^2_{\dd k}(\mathbf{R})$, the estimate in \eqref{sobest} becomes:
\begin{equation}
\| \eta (\cdot, t) \|_{L^2_{\dd x}} \leq c_2 \|\eta_0\|_{H^2_{\dd x}(\mathbf{R})}. \label{estfinal}
\end{equation}
The result now follows from \eqref{estfinal} by linearity.
\end{proof}
This concludes the study of the problem on the full line. We note that the solution involves integrals of the form:
\[ \int_{\R^+} f(k)e^{i(kx-\omega t)}\, \dd k, \]
which are of convenient form for the study of the long time asymptotics of the problem. Indeed, the work in \cite{fokas2005aac} establishes stability results for solutions of evolution PDEs on the half line by analysing precisely these types of integrals. In addition, it is a simple task to evaluate numerically the solution \eqref{solnwhole}, see figure \ref{waveplot}.
\begin{figure}\label{waveplot}
\begin{center}
\includegraphics[scale=0.7]{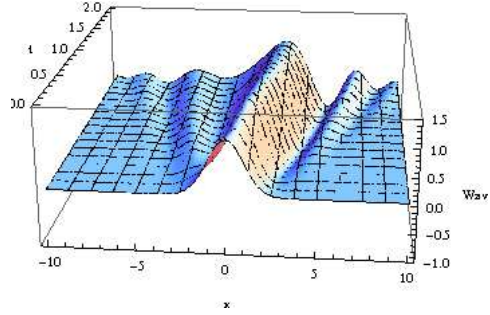}
\caption{Plot of $\eta$ for $(x,t) \in [-10,10]\times [0,2]$ with $\eta_0(x)=\exp \big(\tfrac{-x^2}{2}\big)$.}
\end{center}
\end{figure}

\section{The Semi-Infinite line}
The classical Fourier transform pair on the half-line is
\begin{equation} \eta (x,t) = \frac{1}{\sqrt{2\pi}} \int_\R e^{ikx} \hat{\eta}(k,t)\,\mathrm \mathrm \mathrm dk, \qquad \hat{\eta}(k,t) = \frac{1}{\sqrt{2\pi}}\int_{\R^+} e^{-ikx}\eta(x,t)\, \mathrm dx. \end{equation}
Clearly $\hat{\eta}(k,t)$ is well defined and analytic for $k\in\mathcal{D}_3 \cup \mathcal{D}_4$, where $\mathcal{D}_i$ represents the $i$th quadrant in $\mathbf{C}$. Similarly,
\begin{equation} \varphi (x,t) = \frac{1}{\sqrt{2\pi}} \int_\R  e^{ikx} \hat{\varphi}(k,t)\, \mathrm dk, \qquad \hat{\varphi}(k,t) = \frac{1}{\sqrt{2\pi}}\int_{\R^+} e^{-ikx}\varphi (x,t)\, \mathrm dx. \end{equation}
These definitions will be used extensively in what follows.
\begin{prop}\label{fundk}
Let $\eta(x,t)$ and $\varphi(x,t)$ satisfy \eqref{half_line_eq}. Then the half line Fourier transforms of $\eta(x,t)$ and $\varphi(x,t)$ satisfy the following equation:
\begin{equation} \hat{\varphi} (k,t) = \frac{\hat{\eta}_t (k,t)}{k} + iU \hat{\eta}(k,t), \qquad k \in \mathcal{D}_4 \label{fundkeq}. \end{equation}
\end{prop}
\begin{proof}
Equation \eqref{halfFour} can be written as
\begin{align} 0 &=  k\hat{\varphi}  - \hat{\eta}_t  - iUk \hat{\eta} +k \int_{-\infty}^0 e^{-ikx}\varphi(x,t)\, \dd x  \nonumber \\
 &=  \Phi^+(k) + \Phi^-(k), \label{rhprob}
\end{align}
where $\Phi^+$ $(\Phi^-)$ is analytic in $\mathcal{D}_1$ ($\mathcal{D}_4)$, and integration by parts and the Riemann Lebesgue lemma gives $\Phi^{\pm}=O(1/k)$ at $\infty$. We observe that \eqref{rhprob} defines the jump condition of an elementary Riemann-Hilbert problem with associated contour $\gamma = \mathbf{R}^+$, and the unique solution is $\Phi^+(k)=\Phi^-(k)=0$. We conclude that for $k>0$ we have the following relation:
\[ \hat{\varphi} (k,t) = \frac{\hat{\eta}_t (k,t)}{k} + iU \hat{\eta}(k,t) \]
and by analytic continuation, this extends to $k\in\mathcal{D}_4$.
\end{proof}
\begin{prop}\label{propo}
The initial-boundary value problem stated in (1.1)-(1.6) is equivalent to the following problem:
\begin{alignat}{2} \left(1+\tfrac{1}{k}\right)\hat{\eta}_{tt} + 2iU\hat{\eta}_t + \left( k^4 - U^2k\right) \hat{\eta} &= f(k,t),&\quad &k\in \mathcal{D}_4,\,\, 0<t<T, \label{prop5a} \\ 
\tfrac{1}{k}\hat{\eta}_{tt}(k,0) + 2iU \hat{\eta}_t - kU^2 \hat{\eta} &= 0,& \quad & k\in \mathcal{D}_4, \label{prop5b}\\
\hat{\eta}(k,0) &= \hat{\eta}_0 (k),& \quad & k\in \mathcal{D}_4, \label{prop5c}\end{alignat}
where
\[ f(k,t) \equiv \eta_{xxx}(0,t)+ik\eta_{xx}(0,t).\]
\end{prop}
\begin{proof}
We begin by applying the Fourier transform on $\R^+$ to \eqref{half1}. Integrating by parts several times gives:
\[ \hat{\eta}_{tt} + k^4 \hat{\eta} + \hat{\varphi}_t + ikU\hat{\varphi} = \eta_{xxx}(0,t) + ik\eta_{xx}(0,t), \]
where again we have set $\phi (0,0,t)=0$ without loss of generality. Using the result from proposition \ref{fundk}, and the result from differentiating \eqref{fundkeq}, we find \eqref{prop5a}. Next, we apply the Fourier transform to \eqref{half3}, which gives:
\[ \hat{\varphi}_t (k,0) + ikU\hat{\varphi}(k,0) = 0. \]
Using \eqref{fundkeq} as well as the time derivative of \eqref{fundkeq} we find \eqref{prop5b}. Finally, \eqref{prop5c} follows from the application of the Fourier transform to the initial data in \eqref{initial}.
\end{proof}
\begin{prop}
The solution to the IBVP presented in proposition \ref{propo} is given by:
\begin{align} \hspace{-4mm}\hat{\eta} (k,t) &= \left [ c_+ (k) + \alpha (k)F_t(\omega_+,k) \right ]e^{-i\omega_+ t} + \left [ c_- (k) - \alpha (k)F_t(\omega_-,k) \right ]e^{-i\omega_- t}, \nonumber\\
 & \hspace{7.5cm} k\in\mathcal{D}_4,\,\, 0<t<T, \label{soln} \end{align}
where $c_\pm$ and $\omega_\pm$ are defined in \eqref{cpm} and \eqref{omegapm} and $F_t\{\omega (k),k\}$ is defined by:
\begin{equation} F_t\left\{ \omega (k),k\right\} = \int_0^t
  e^{i\omega(k)\tau}\left [
    \eta_{xxx}(0,\tau)+ik\eta_{xx}(0,\tau)\right ]\, \mathrm d\tau,\quad k\in\mathbf{C}, \,\, 0<t<T. \end{equation}
This will be referred to as the global relation \cite{fokas1997utm}.
\end{prop}
\begin{proof}
Differentiating and using the fundamental theorem of calculus we find the RHS of \eqref{soln} solves \eqref{prop5a} and satisfies \eqref{prop5b}-\eqref{prop5c}.
\end{proof}
We have now obtained a solution to the IBVP in \eqref{half_line_eq}, but in terms of the \emph{unknown} functions $\eta_{xxx}(0,t)$ and $\eta_{xx}(0,t)$. These functions will be determined in \S 4.

\begin{remark}For a number of IBVPs it is possible to eliminate the \emph{transforms} of the unknown boundary values using only \emph{algebraic} manipulations \cite{fokas1997utm}. This approach utilises the analytic dependence on $k$, hence it suggests that we should re-parameterise the spectral problem, so that $\omega$ takes a simpler form. However, there does \emph{not} exist a rational re-parameterisation $k=k(t)$, $\omega = \omega(t)$. Indeed, let us seek a rational parameterisation to the equation $P(X,Y;U)=0$, where
\[ P(X,Y;U)\defn XY^2 - X^5 + (Y-UX)^2 \]
This equation defines\footnote{Here we assume $U$ is rational, which is without loss of generality, since $U$ corresponds to a physical constant and can be approximated arbitrarily by rationals.} an algebraic curve over $\mathbf{Q}$, and it is possible to show that the genus of this curve, $g(P)$, is greater than zero for all but a few (unphysical) values of the parameter $U$. With this result, Falting's theorem (see below) implies that there is \emph{no} rational parameterisation.
\begin{theorem*}[Faltings \cite{faltings1984fta}]Suppose $C$ is a non-singular algebraic curve over $\mathbf{Q}$ with genus $g(C)$. If $g(C)>0$ then only a finite number of rational points lie on $C$.\end{theorem*}
\noindent
An immediate corollary is that for $g(P)>0$, the equation $P(X,Y;U)=0$ has \emph{no} rational parameterisation.
\end{remark}

\section{Determination of the Unknown Boundary Values}
The solution $\eta (x,t)$ requires that $\hat{\eta}(k,t)$ is solved for $k>0$. However, the global relation \eqref{soln} is valid for $k$ in a much larger domain, namely $k\in \mathcal{D}_4$. It turns out that this extra freedom allows us to determine the unknown boundary values $\eta_{xxx}(0,t)$ and $\eta_{xx}(0,t)$.

Our analysis involves the complex $k$-plane, so we must first choose appropriate branches for the functions $\omega_\pm (k)$. It is clear that the branch points of $\omega_\pm (k)$ are at $k=0$ and at the three roots of the cubic equation
\[ k^3 + k^2 -U^2 =0. \]
Since $U$ is real, two of the roots are a complex conjugate pair and the remaining root is real. The loci of these points is shown in Figure 2 with appropriate branch cuts.
\begin{figure}\label{locicuts}
\begin{center}
\includegraphics[scale=0.17]{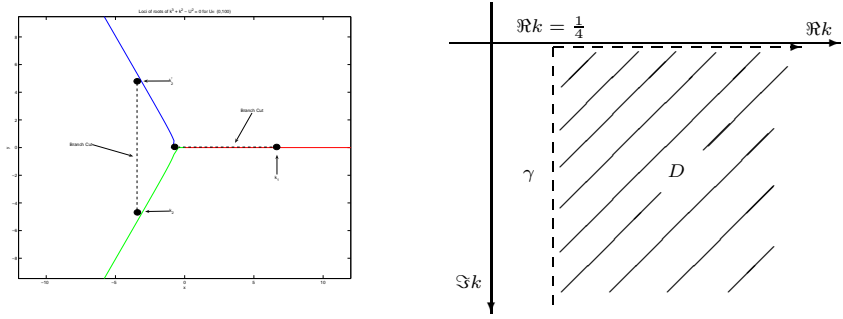}
\hspace{0.5cm}
\setlength{\unitlength}{2279sp}%
\begingroup\makeatletter\ifx\SetFigFontNFSS\undefined%
\gdef\SetFigFontNFSS#1#2#3#4#5{%
  \reset@font\fontsize{#1}{#2pt}%
  \fontfamily{#3}\fontseries{#4}\fontshape{#5}%
  \selectfont}%
\fi\endgroup%
\begin{picture}(4269,3399)(-461,-2113)
\thinlines
{\color[rgb]{0,0,0}\put(  1,1274){\vector( 0,-1){3375}}
}%
{\color[rgb]{0,0,0}\put(-449,839){\line( 1, 0){4035}}
\put(3586,839){\vector( 1, 0){210}}
}%
{\color[rgb]{0,0,0}\multiput(676,-2011)(0.00000,207.77778){14}{\line( 0, 1){103.889}}
\multiput(676,794)(214.80000,0.00000){13}{\line( 1, 0){107.400}}
\put(3361,794){\vector( 1, 0){ 15}}
}%
{\color[rgb]{0,0,0}\put(766,359){\line( 1, 1){375}}
}%
{\color[rgb]{0,0,0}\put(751,-106){\line( 1, 1){855}}
}%
{\color[rgb]{0,0,0}\put(751,-511){\line( 1, 1){1260}}
}%
{\color[rgb]{0,0,0}\put(751,-961){\line( 1, 1){1695}}
}%
{\color[rgb]{0,0,0}\put(751,-1441){\line( 1, 1){2175}}
}%
{\color[rgb]{0,0,0}\put(766,-1876){\line( 1, 1){1080}}
}%
{\color[rgb]{0,0,0}\put(3226,584){\line(-1,-1){915}}
}%
{\color[rgb]{0,0,0}\put(1291,-1876){\line( 1, 1){1785}}
}%
{\color[rgb]{0,0,0}\put(1921,-1876){\line( 1, 1){1155}}
}%
{\color[rgb]{0,0,0}\put(2581,-1876){\line( 1, 1){495}}
}%
\put(3421,929){\makebox(0,0)[lb]{\smash{{\SetFigFontNFSS{7}{8.4}{\familydefault}{\mddefault}{\updefault}{\color[rgb]{0,0,0}$\Re k$}%
}}}}
\put(-389,-1831){\makebox(0,0)[lb]{\smash{{\SetFigFontNFSS{7}{8.4}{\familydefault}{\mddefault}{\updefault}{\color[rgb]{0,0,0}$\Im k$}%
}}}}
\put(271,959){\makebox(0,0)[lb]{\smash{{\SetFigFontNFSS{7}{8.4}{\familydefault}{\mddefault}{\updefault}{\color[rgb]{0,0,0}$\Re k=\tfrac{1}{4}$}%
}}}}
\put(346,-631){\makebox(0,0)[lb]{\smash{{\SetFigFontNFSS{7}{8.4}{\familydefault}{\mddefault}{\updefault}{\color[rgb]{0,0,0}$\gamma$}%
}}}}
\put(1921,-616){\makebox(0,0)[lb]{\smash{{\SetFigFontNFSS{7}{8.4}{\familydefault}{\mddefault}{\updefault}{\color[rgb]{0,0,0}$D$}%
}}}}
\end{picture}%
\caption{Loci of branch points (left). The contour $\gamma$ and domain $D$ (right).}
\end{center}
\end{figure}
\begin{prop}\label{propmain}
Let $\hat{\eta}(k,t)$ satisfy the global relation \eqref{soln} for $k\in\mathcal{D}_4$. Then the unknown functions $\eta_{xxx}(0,t)$ and $\eta_{xx}(0,t)$ satisfy the equation
\begin{equation} 2\pi i \eta_{xx}(0,t) = g(t) - \int_0^t K(t-\tau) \eta_{xxx}(0,\tau)\, \dd \tau, \label{Fredholm} \end{equation}
where the functions $g(t)$ and $K(t)$ are defined as follows:
\[ g(t) = \int_\gamma \frac{c_- (k)}{\alpha (k)} e^{-i\omega_-(k)t}\, \dd \mu(k), \qquad K(t) = \int_\gamma e^{-i\omega_-(k)t}\, \dd \mu(k). \]
The contour $\gamma$ is defined by $\gamma  = (-i\infty + \tfrac{1}{4}, \tfrac{1}{4}] \cup [\tfrac{1}{4}, \infty)$ and $\dd \mu =\tfrac{1}{k}\dd \omega_-(k)$ defines an analytic measure \cite{forelli1963am} on $\gamma$.
\end{prop}
\begin{proof}
We evaluate \eqref{soln} at $t=T$, and multiply the resulting expression by
\[ \frac{ e^{i\omega_- (T-t)}}{k\alpha (k)} \left( \frac{\dd \omega_-}{\dd k}\right) \]
and then we integrate with respect to $\dd k$ along $\gamma$:
\begin{multline}
 \int_{\gamma }\frac{\hat{\eta}(k,T)}{\alpha (k)}e^{i\omega_- (T-t)}\, \dd\mu  - \int_{\gamma } \left [\frac{c_+(k)}{\alpha(k)} + F_T(\omega_+,k) \right ]e^{-i(\omega_+-\omega_-) T}e^{-i\omega_- t}\, \dd\mu  \\
  =\int_{\gamma } \left [\frac{ c_- (k)}{\alpha(k)} -  F_T(\omega_-,k) \right ]e^{-i\omega_- t} \, \dd\mu.
\end{multline}
The integrand of the integral on the RHS of this equation is bounded and analytic in the domain bounded $D\subset \mathcal{D}_4$ defined by:
\[ D= \{ k\in \mathcal{D}_4: \Re k \geq \tfrac{1}{4} \}. \]
Indeed, this is a consequence of the following two facts:
\begin{enumerate}
 \item The following estimates hold for $k\rightarrow \infty$ in $D$:
\[\frac{c_- (k)}{k\alpha (k)}\left( \frac{\dd \omega_-}{\dd k}\right) =  O\left( \frac{1}{k}\right), \quad \frac{\hat{\eta}(k,T)}{k\alpha(k)} \left( \frac{\dd \omega_-}{\dd k}\right) =  O\left( \frac{1}{k} \right).\]
The determination of these asymptotic estimates uses integration by parts, the Riemann-Lebesgue lemma and the assumptions $\hat{\eta}_0 \in H^2_{\langle 2\rangle}(\mathbf{R}^+)$ and $\eta (0,T)=\eta_x(0,T)=0$.
 \item For $|k|$ sufficiently large, $t>0$ we have:
\begin{align*}
|\exp (-i\omega_+ t)| &< \exp \left\{ \left(\epsilon+\tfrac{1}{2}k_I[4k_R-1]\right)t\right\} \\
|\exp (-i\omega_- t)| &< \exp \left\{ \left(\epsilon-\tfrac{1}{2}k_I[4k_R-1]\right)t\right\}
 \end{align*}
for $\epsilon>0$. The determination of these estimates is based on the asymptotic formulae:
\begin{align*}
 \omega_+ (k) &= +k^2 - \tfrac{1}{2}k + (U+\tfrac{3}{8}) + o(1), \qquad k\rightarrow \infty, \\
 \omega_- (k) &= -k^2 + \tfrac{1}{2}k + (U-\tfrac{3}{8}) + o(1), \qquad k\rightarrow \infty.
\end{align*}
Thus for $|k|$ sufficiently large we have $| \Im \omega_+ (k) - \Im\left( k^2 -\tfrac{1}{2}k\right) | < \epsilon$.
\end{enumerate}
\begin{figure}
\begin{center}
\includegraphics[scale=0.25]{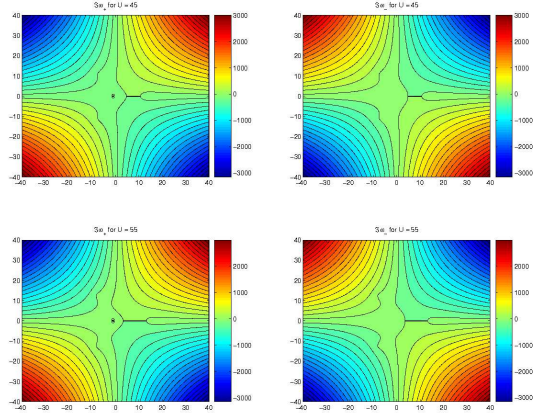}
\caption{Contour plots of $\Im \omega_+$ and $\Im \omega_-$ for two values of $U$.}
\end{center}
\end{figure}
Using these facts as well as the equation
\[ \int_\gamma e^{-i\omega_-(k)(t-\tau)}\, \dd \omega_-(k) = 2\pi \delta (t-\tau), \]
where the equality is understood in the distributional sense, we find:
\[0 =  \int_{\gamma }\frac{ c_- (k)}{\alpha(k)}e^{-i\omega_- t} \, \dd\mu - \int_0^T\!\! \int_\gamma  e^{i\omega_- (\tau-t)} \eta_{xxx}(0,\tau)\,\dd\mu\,\dd \tau - 2\pi i \eta_{xx}(0,t). \]
Finally, we observe that:
\[ \int_{\gamma }\int_t^T e^{-i\omega_- (t-\tau)} \eta_{xxx}(0,\tau)\,\mathrm d\tau\, \dd\mu = 0, \]
which follows from the fact that $t-\tau<0$ in the integrand, so the integrand is bounded and analytic on the RHS of $\gamma$ (i.e in $D$), and them Cauchy's theorem implies that the above integral vanishes.\end{proof}
\begin{lem}\label{integralop}
Let $K(t)$ and $g(t)$ be defined by Proposition \ref{propmain} and let $\Sigma\subset \mathbf{R}^+$ be a bounded interval. Then:
\begin{enumerate}[(a)]
 \item The integral operator $K$ defined by:
\[ Kf \defn \int_0^t K(t-\tau) f(\tau)\, \dd \tau, \qquad f \in L^2_{\dd t}(\Sigma)\]
has a weak singularity.
\item $g \in L^2_{\dd t}(\Sigma)$.
\end{enumerate}
\end{lem}
\begin{proof}
For $(a)$ it is enough to prove:
\[ | K(t,\tau) | \leq M | t-\tau |^{-\alpha} \]
for some $\alpha \in (0,1)$. We show this as follows: choose $\delta>1$ such that the boundary $\partial B_\delta$, where $B_{\delta} \equiv B_\delta (0)$, lies beyond the branch cut between the complex pair of roots of $k^3+k^2-U^2=0$. The contribution from within $B_\delta$ is certainly uniformly bounded, since $\gamma$ has been chosen to avoid the singularities of the relevant functions. This gives us the following inequality:
\[ \left| K(t,\tau)\right| < c_0 + \left| \int_{\gamma \setminus (\gamma \cap B_\delta)} e^{-i\omega_- (k)(t-\tau)}\, \dd \mu \right|. \]
Now split the contour into $\gamma_1$ and $\gamma_2$ which constitute the disjoint elements of $\gamma \setminus (\gamma \cap B_\delta)$: the first is parallel with the negative imaginary axis, the second is parallel with the positive real axis. Choosing $\delta >0$ sufficiently large, we can use the result in Proposition \ref{propmain} so that:
\[ \left|\omega_- (k) + k^2 - \tfrac{1}{2} k - U + \tfrac{3}{8}\right| < \epsilon \]
on $\gamma_1$ and $\gamma_2$. Noting that the integrand decays exponentially in the 1st and 3rd quadrants, we deform the $\gamma_1$ and $\gamma_2$ onto the rays $\mathrm{arg} (k) = \tfrac{\pi}{4}$ and $\mathrm{arg}(k) = - \tfrac{3\pi}{4}$ respectively, picking up a contribution from the relevant parts of $\partial B_\delta$. These additional contributions will also be uniformly bounded. Denoting the partial rays by $\Gamma_1$ and $\Gamma_2$ respectively gives:
\[ \left| \int_{\gamma_1} e^{-i\omega (k)(t-\tau)}\, \dd\mu\right | < c_1 + \int_{\Gamma_1}|e^{-i\omega_- (k)(t-\tau)}|\, |\dd \mu|. \]
Once again using the result in Proposition \ref{propmain} we have:
\begin{align*}
\int_{\Gamma_1}|e^{-i\omega_- (k)(t-\tau)}|\, |\dd \mu| &\leq \int_{\delta}^\infty e^{\epsilon (t-\tau)} e^{(-r^2 + 2r)(t-\tau)}\left ( \epsilon + 2\right)\, \dd r \\
&\leq c_2 e^{(\epsilon+1) (t-\tau)} \int_1^\infty e^{-(r-1)^2(t-\tau)}\, \dd r \\
&\leq c_3 \frac{e^{(\epsilon+1) (t-\tau)}}{\sqrt{(t-\tau)}} \\
&< c_4 |t-\tau|^{-1/2}.
\end{align*}
Using an entirely analogous argument for the contribution from $\Gamma_2$ we find:
\[ \int_{\Gamma_2}|e^{-i\omega_- (k)(t-\tau)}|\, |\dd \mu| <c_5 |t-\tau|^{-1/2}. \]
Each of the previous bounds implies the following result:
\begin{align*} |K(t,\tau)| & \leq c_6 + c_7|t-\tau|^{-1/2} \\
 &= \left(c_6|t-\tau|^{1/2} + c_7\right) |t-\tau|^{-1/2} \\
&< M |t-\tau|^{-1/2}. 
\end{align*}
Hence $K$ has a weak singularity with exponent $\alpha = \tfrac{1}{2}$.

For $(b)$, we proceed as in the previous case by splitting the integral. Noting the asymptotic result in Proposition \ref{propmain}, for $\delta$ sufficiently large, we have:
\[ \left| \frac{c_-(k)}{\alpha(k)}\right| <\frac{\epsilon}{|k|} \]
for $k\in \gamma\setminus (\gamma \cap B_\delta)$, where $\epsilon >0$. This gives:
\begin{align*} \left| \int_{\gamma\setminus (\gamma\cap B_\delta)} \frac{ c_- (k)}{\alpha(k)}  e^{-i\omega_-(k)t}\, \dd \mu\right| &\leq \epsilon \left| \int_{\gamma\setminus (\gamma\cap B_\delta)} \frac{ e^{-i\omega_-(k)t}}{|k|}\, \dd \mu\right| \\
&\leq c_1 \int_2^\infty r^{-1} e^{-(r-1)^2t}\, \dd r \\
&< c_2 \log t 
\end{align*}
where we have used similar estimates as earlier. Note also the singularities of the integrand lie off $\gamma$ by construction, so the contribution from the integral along $\gamma\cap B_\delta$ is uniformly bounded and we have:
\[ \left| \int_\gamma \frac{c_-(k)}{\alpha(k)} e^{-i\omega_-(k)t}\, \dd \mu \right| < c_3 \log t. \]
It then follows that for $\Sigma$ bounded, we have $\|g\|_{L^2_{\dd t}(\Sigma)} < \infty$.
\end{proof}
\begin{corollary}
The operator $K: L^2_{\dd t}(\Sigma) \rightarrow L^2_{\dd t}(\Sigma)$ is compact.
\end{corollary}
\begin{proof}
$K$ has a weak singularity, hence it is bounded in $L^2_{\dd t}(\Sigma)$. Indeed, take $f \in L^2_{\dd t}(\Sigma)$, then the previous result gives:
\begin{align*}|(Kf)(t)|^2 &\leq \left|\int_{\Sigma}|K(t,\tau)|^{1/2} |K(t,\tau)|^{1/2} |f(\tau)|\, \dd \tau \right|^2 \\
&\leq \left| \int_\Sigma |K(t,\tau)||f(\tau)|^2\, \dd \tau\right| \left| \int_{\Sigma}|K(t,\tau)|\, \dd \tau\right| \\
&\leq M_1\left| \int_\Sigma |K(t,\tau)||f(\tau)|^2\, \dd \tau\right|,
\end{align*}
which follows from the Cauchy-Schwarz inequality and the fact that a weak singularity is integrable. Thus,
\begin{align*}
\|Kf\|_{L^2_{\dd t}(\Sigma)}^2 &\leq M_1\int_{\Sigma \times \Sigma} |K(t,\tau)||f(\tau)|^2\, \dd \tau\, \dd t \\
&\leq M_2 \|f\|^2_{L^2_{\dd t}(\Sigma)}
\end{align*}
where we have used Fubini's theorem. Thus, standard results (see for example \cite{kantorovich1964fan}), imply the desired result.
\end{proof}
\begin{remark}
The estimates of the above lemmas are not sufficient to handle the case $T=\infty$, i.e. $\Sigma=(0,\infty)$, since in thise case the contributions from $\gamma \cap B_\delta$ fail to decay rapidly for large $t$. This underlines the fact that instabilities could grow exponentially in time and hence analysis on $L^2_{\dd t}(\mathbf{R}^+)$ is inappropriate. The lemma also indicates the importance of the conditions on the curvature of the plate at the hinge, since without them we cannot require that $g\in L^2_{\dd t}(\mathbf{R}^+)$.
\end{remark}
Since the integral operator $K$ is of the convolution type, it follows that \eqref{Fredholm} can be solved in a straightforward manner using Laplace's transform. In this respect the following estimates, which can be derived using similar methodology to that used in lemma \ref{integralop}, are essential.
\begin{lem}
Given $g$ and $K$ as in Proposition \ref{propmain}, $\exists \alpha>0$ such that $e^{-\alpha t}g\in L^1_{\dd t}(\R^+)$ and $e^{-\alpha t} K \in L^1_{\dd t} (\R^+)$.
\end{lem}
The solution of \eqref{Fredholm} is unique provided that $g \in \mathrm{dom}(K^{-1})$. The compactness of $K$ means that $K^{-1}$ is unbounded and hence necessarily discontinuous on $L^2_{\dd t}(\Sigma)$. This fact suggests the existence of a sequence of initial data, $\{\eta_0^n\} \in H^2_{\langle 2\rangle}(\mathbf{R}^+)$, such that the corresponding sequence $g_n\in L^2_{\dd t}(\Sigma)$, defined by:
\[ g_n(t) = \int_\gamma \left[ \frac{\tilde{c}_-(k)}{\alpha(k)}\right] \hat{\eta}_0^n(k)e^{-i\omega_-(k)t}\, \dd \mu, \]
has the following two properties:
\begin{align*}
g_n &\rightarrow 0, \\
\|K^{-1}g_n\|_{L^2_{\dd t}(\Sigma)} &\rightarrow M >0.
\end{align*}
Now since $\hat{\eta}(k,t)$ has non-pathological dependence on $\eta_{xxx}(0,t)$, and hence $K^{-1}g$, it would follow that $\hat{\eta}(k,t)$ would change discontinuously with the initial data. Hence in this case the problem would be ill-posed in $L^2_{\dd x}(\mathbf{R}^+)$.

We now perform a regularisation of the problem by modifying the relevant function spaces so that $K$ is continuously invertible, which will ensure well-posedness. A key result, central to our argument, is the following well-known lemma.
\begin{lem}\label{tik}
Let $X,Y$ be Banach spaces and let $X_c \subset X$ be compact. Then if $T:X_c \rightarrow Y$ is continous and one-to-one, then $T^{-1}:\mathcal{R}(X_c)\rightarrow X_c$ exists and is continuous.
\end{lem}
Since $K:L^2_{\dd t}\rightarrow L^2_{\dd t}$ is compact, it is necessarily bounded (continuous), so it follows immediately from lemma \ref{tik} that if we choose $X_c$ to be compactly embedded in $L^2_{\dd t}(\Sigma)$ such that $K|_{X_c}$ is one-to-one, then the following map a homeomorphism:
\[ K|_{X_c}:X_c \rightarrow \mathcal{R}(K|_{X_c}), \]
where $K|_{X_c}$ denotes the restriction of $K$ to $X_c$. We now prove that $K|_{X_c}$ is one-to-one.
\begin{lem}\label{one2one}
The integral operator $K|_{X_c}:X_c \rightarrow L^2_{\dd t}(\Sigma)$ is one-to-one.
\end{lem}
\begin{proof}
It suffices to prove that $\mathcal{N}(K|_{X_c})=\{0\}$ since $K|_{X_c}$ is linear. Suppose $\theta \in \mathcal{N}(\Sigma)$, so that:
\begin{equation} K|_{X_c} \theta = 0. \label{ker} \end{equation}
Since $K|_{X_c}$ is of the convolution type equation \eqref{ker} can be written as:
\begin{equation} (K * \theta)(t)=0, \qquad t\in\Sigma \label{ker2} \end{equation}
where $K(t)=\int_\gamma e^{-i\omega_-(k) t} \dd \mu$. Now it follows from Titchmarsh's convolution theorem that $\theta=0$ almost everywhere in $(0,t_1)$ and $K=0$ almost everywhere in $(0,t_2)$ where, $t_1 + t_2\geq T$. The operator $K$ defines an analytic function of $t$ for $t\in \Sigma$, and as such the zero set for $K$ has measure zero. Consequently $\theta=0$ almost everywhere in $\Sigma$ and we conclude that $\mathcal{N}(K|_{X_c}) =\{0\}$, so $K|_{X_c}$ is one-to-one.
\end{proof}
\begin{corollary}
The map $K|_{X_c}:X_c \rightarrow \mathcal{R}(K|_{X_c})$ is a homeomorphism.
\end{corollary}
Now recall the classical result due to Rellich and Kondrachov, which states that $H^1_{\dd t}(\Sigma)$ is compactly embedded in $L^2_{\dd t}(\Sigma)$, so our previous discussion reveals that the map:
\[ K|_{H^1_{\dd t}(\Sigma)} : H^1_{\dd t}(\Sigma) \rightarrow \mathcal{R}(K|_{H^1_{\dd t}(\Sigma)}) \]
is continously invertible. It now suffices to choose $\eta_0$ so that $g \in \mathcal{R}(K|_{H^1_{\dd t}(\Sigma)})$, where $g$ is defined in Propostion \ref{propmain}. In fact, we can simply choose $\eta_0$ so that $g \in H^1_{\dd t}(\Sigma)$ since the Fredholm alternative theorem implies:
\[ \mathcal{R}(K|_{H^1_{\dd t}(\Sigma)}) = \mathcal{N}(K|_{H^1_{\dd t}(\Sigma)}^*)^{\perp}  \]
where the $K^*$ denotes the adjoint operator. Since $K|_{H^1_{\dd t}(\Sigma)}^*$ is also a convolution operator, a similar argument to that used in the proof of Lemma \ref{one2one} shows that $\mathcal{N} (K|_{H^1_{\dd t}(\Sigma)}^*)=\{0\}$, so we conclude that:
\[ \mathcal{R}(K|_{H^1_{\dd t}(\Sigma)}) = H^1_{\dd t}(\Sigma). \]
For classical solutions we also require $g(0)=0$. We can now give sufficient conditions of $\eta_0$ so IBVP in \eqref{half_line_eq} well posed.
\begin{prop}
If $\eta_0 \in H^3_{\langle 3\rangle }(\mathbf{R}^+)$, then $g \in H^1_{\dd t}(\Sigma)$ and the IBVP \eqref{half_line_eq} is well posed, i.e the solution defined by $\eta_0 \mapsto \eta$ defines a continous map from $H^3_{\langle 3\rangle}(\mathbf{R}^+)$ to $L^2_{\dd x}(\mathbf{R}^+)$.
\end{prop}
\begin{proof}[Sketch Proof]
First we note that for $\eta_0 \in H^3_{\langle 3\rangle}(\mathbf{R}^+)$ we have:
\[ \left| \frac{ c_-(k)}{\alpha (k)}\right| < \frac{ \epsilon}{|k|^3} \]
for $|k|$ sufficiently large. This follows from estimates similar to those in Proposition \ref{propmain}. In this case an application of Cauchy's theorem gives that $g(0)=0$, as required for the equation $Kf=g$ to have a solution in a classical sense. Also, for $\eta_0 \in H^3_{\langle 3\rangle}(\mathbf{R}^+)$, the map defined by:
\[ G: H^3_{\langle 3 \rangle}(\mathbf{R}^+) \rightarrow L^2_{\dd t}(\Sigma): \eta_0 \mapsto g, \]
is continuous. In addition, using estimates similar to those in the proof of Lemma \ref{integralop} we find $g'(t)\in L^2_{\dd t} (\Sigma)$, hence $g \in H^1_{\dd t}(\Sigma)$ and it follows from our previous discussion that the solution to the integral equation depends continuously on $g$ with respect to the $L^2_{\dd t}(\Sigma)$ topology. It follows from estimates similar to those in Proposition \ref{firstwellposed} that the solution map $S_t:H^3_{\langle 3\rangle}(\mathbf{R}^+)\rightarrow L^2_{\dd x}(\mathbf{R}^+)$, defined by $\eta_0 \mapsto \eta$, is continuous.
\end{proof}

\nocite{crighton1991flm}
\nocite{kato1995ptl}
\nocite{ablowitz2003cvi}
\nocite{whitham1974lan}
\nocite{rudin1993rac}
\bibliographystyle{plain}
\bibliography{../../../../Bibliography/ants_bib}

\begin{thebibliography}{10}

\bibitem{ablowitz2003cvi}
M.J. Ablowitz and A.S. Fokas.
\newblock {\em Complex Variables: Introduction and Applications}.
\newblock Cambridge University Press, 2003.

\bibitem{antnlflp}
A.C.L. Ashton.
\newblock A pseudodifferential approach to fluid loaded membranes.
\newblock {\em (preprint)}, 2009.

\bibitem{crighton1991flm}
DG~Crighton and JE~Oswell.
\newblock Fluid loading with mean flow. i. response of an elastic plate to
  localized excitation.
\newblock {\em Philosophical Transactions: Physical Sciences and Engineering},
  335(1639):557--592, 1991.

\bibitem{faltings1984fta}
G.~Faltings.
\newblock Finiteness theorems for abelian varieties over number fields.
\newblock {\em Arithmetic geometry, Pap. Conf., Storrs/Conn.}, 1984.

\bibitem{fokas1997utm}
A.S. Fokas.
\newblock A unified transform method for solving linear and certain nonlinear
  pdes.
\newblock {\em Proceedings: Mathematical, Physical and Engineering Sciences},
  453(1962):1411--1443, 1997.

\bibitem{fokas2008uab}
A.S. Fokas.
\newblock {\em {A Unified Approach to Boundary Value Problems}}.
\newblock CBMS 78, SIAM, 2008.

\bibitem{fokas2005aac}
AS~Fokas and DT~Papageorgiou.
\newblock Absolute and convective instability for evolution pdes on the
  half-line.
\newblock {\em Studies in Applied Mathematics}, 114(1):95--114, 2005.

\bibitem{forelli1963am}
F.~Forelli.
\newblock Analytic measures.
\newblock {\em Pacific J. Math}, 13(2):571--578, 1963.

\bibitem{kantorovich1964fan}
L.V. Kantorovich and G.P. Akilov.
\newblock {\em Functional Analysis in Normed Spaces}.
\newblock Pergamon Press, 1964.

\bibitem{kato1995ptl}
T.~Kato.
\newblock {\em Perturbation Theory for Linear Operators}.
\newblock Springer, 1995.

\bibitem{rudin1993rac}
W.~Rudin.
\newblock {\em Real and Complex Analysis}.
\newblock 1993.

\bibitem{whitham1974lan}
G.B. Whitham.
\newblock {\em Linear and nonlinear waves}.
\newblock New York, 1974.

\end{thebibliography}

\end{document}